\documentclass[12pt, showkeys]{amsart}%
\usepackage{color}
\usepackage{graphicx,subfigure}
\usepackage{subfigure}
\usepackage{subfigure}
\usepackage{amssymb}
\usepackage{amsmath,amsxtra}
\usepackage{multirow}
\usepackage{enumerate}
\usepackage{epstopdf}
\usepackage{cite,stmaryrd,txfonts}
\usepackage{tikz}
\usepackage{pgf}
\usetikzlibrary{snakes}
\usetikzlibrary{angles,calc,patterns}
\usetikzlibrary{shapes,arrows,automata}
\usepackage{bbm}
\usepackage{dsfont}
\usepackage{accents}
 \usepackage{float}

\usepackage{epic}
\usepackage{empheq}
\usepackage{cases}
\usepackage{diagbox}
\def\dx{\,{\rm dx}}

\newtheorem{theorem}{Theorem}[section]
\newtheorem{remark}[theorem]{Remark}
\newtheorem{proposition}[theorem]{Proposition}
\newtheorem{lemma}[theorem]{Lemma}

\newcounter{mnote}
\let\oldmarginpar\marginpar
\renewcommand\marginpar[1]{\-\oldmarginpar[\raggedleft\footnotesize #1]
  {\raggedright\footnotesize #1}}

\usepackage{footnote}
\usepackage{booktabs}

\usepackage{rotating}

\numberwithin{equation}{section}

\usepackage{cite}

\usepackage[colorlinks,linkcolor=black,
anchorcolor=black,
citecolor=black]{hyperref}
\usepackage[shortlabels]{enumitem}
\setlist[enumerate]{nosep}
\usepackage{color, soul}
\soulregister\cite7
\usepackage{setspace}

\def\rn{\boldsymbol{\mathbb{R}^{\boldsymbol{\mathsf{n}}}}}

\def\uv{\undertilde{v}}
\def\uV{\undertilde{V}}

\def\utau{\undertilde{\tau}}

\def\dv{{\rm div}}

\def\od{\mathbf{d}}
\def\oT{\mathbf{T}}

\def\odelta{\boldsymbol{\delta}}
\def\okappa{\boldsymbol{\kappa}}

\def\R{\mathcal{R}}
\def\N{\mathcal{N}}

\usepackage[mathscr]{eucal}

\def\xD{\boldsymbol{\mathbf{D}}}

\def\xv{\boldsymbol{\mathbf{v}}}

\def\fomega{\boldsymbol{\omega}}
\def\fmu{\boldsymbol{\mu}}

\def\feta{\boldsymbol{\eta}}

\def\fpsi{\boldsymbol{\psi}}

\def\fW{\boldsymbol{W}}

\def\ff{\boldsymbol{f}}

\def\ixalpha{\boldsymbol{\alpha}} 
\def\ixbeta{\boldsymbol{\beta}}

\def\ixtau{\boldsymbol{\tau}}
\def\ixsigma{\boldsymbol{\sigma}}

\def\q{\mathcal{Q}}

\begin{document}


\title{Finite element spaces by Whitney $k$-forms on cubical meshes}


\author{Shuo Zhang}
\address{LSEC, Institute of Computational Mathematics and Scientific/Engineering Computing, Academy of Mathematics and System Sciences, Chinese Academy of Sciences, Beijing 100190; University of Chinese Academy of Sciences, Beijing, 100049; People's Republic of China}
\email{szhang@lsec.cc.ac.cn}

\thanks{The research is partially supported by NSFC (12271512, 11871465).}

\begin{abstract}

Finite element spaces by Whitney $k$-forms on cubical meshes in $\rn$ are presented. Based on the spaces, compatible discretizations to $H\Lambda^k$ problems are provided, and discrete de Rham complexes and commutative diagrams are constructed.
\end{abstract}



\maketitle


%
%
%
\section{Introduction}

As main objectives of finite element exterior calculus, there have been many kinds of finite element spaces for differential $k$-forms, for which we refer to, e.g., \cite{Arnold.D;Falk.R;Winther.R2010bams,Arnold.D;Falk.R;Winther.R2006acta,Arnold.D2018feec}. Most of these existing spaces are on simplicial meshes, two representative families being the $\mathcal{P}_r\Lambda^k$ family and the $\mathcal{P}^-_r\Lambda^k$, or trimmed $\mathcal{P}_r\Lambda^k$, family, and among these finite elements, those by Whitney forms are of the minimal shape function spaces. Meanwhile, finite element spaces based on cubical meshes of n-dimensional boxes (products of intervals) are also popular. Families of finite element spaces which are in many ways parallel to the $\mathcal{P}_r\Lambda^k$ and the $\mathcal{P}^-_r\Lambda^k$ families, namely the tensor product family, $\mathcal{Q}^-_r\Lambda^k$, and the trimmed and non-trimmed serendipity families, $\mathcal{S}^-_r\Lambda^k$ and $\mathcal{S}_r\Lambda^k$, have been established. We refer to, e.g., \cite{Arnold.D;Boffi.D;Bonizzoni.F2015finite,Arnold.D;Awanou.G2014,Gillette.A;Hu.K;Zhang.S2020,Lohi.J2023new,Gillette.A;Kloefkorn.T2019trimmed,Cockburn.B;Qiu.W2015,Christiansen.S;Gillette.A2016} for these and related works. We note that, Whitney forms are of the minimal degree in theory to be sufficient for local approximation in $H\Lambda^k$ norms to functions, but among existing finite elements on cubical meshes, shape functions beyond Whitney forms always have to be used. It is natural to investigate if compatible finite element spaces for $H\Lambda^k$ can be constructed by no more than piecewise Whitney forms on cubical meshes. In this paper, we present a confirming answer to this question by constructing a sequence of finite element function spaces and further discrete de Rham complexes thereon. 
~\\

A first attempt to this question used to be made on $H\Lambda^0$ problems, where the task is to construct finite element spaces on cubical meshes by piecewise $\mathcal{P}_1$ polynomial shape functions, rather than $\mathcal{Q}_1$ polynomials which are usually used. An innovative answer to this problem is given by Park-Sheen \cite{Park.C;Sheen.D2003}, where an integral continuity across the interfaces between cells are used to formulate the nonconforming global space of finite element functions. The finite element spaces given in \cite{Park.C;Sheen.D2003}, denoted by $V^{\rm PS}_h$ below, can be naturally viewed an extension of the lowest-degree Crouzeix-Raviart element to cubical meshes. Notably, $V^{\rm PS}_h$ does not coincide with a ``finite element" in Ciarlet's sense \cite{Ciarlet.P1978book}, though, its capacity still admits compatible discretization to $H^1$-elliptic problems. Later, Hu-Shi\cite{Hu.J;Shi.Z2005} discovers that, the space $V^{\rm PS}_h$ contains a subspace by interpolating the conforming bilinear element space to the nonconforming rotated $Q_1$ element space. This provides a remarkable insight to the approximation capacity of $V^{\rm PS}_h$. So far, finite element space $V^{\rm PS}_h$ has been widely applied and studied associated with $H\Lambda^0$ problems. We note that, in spite of this early example, finite element spaces by Whitney forms for $H\Lambda^k$ with general $k$, indeed any $k\geqslant 1$, have not been known to us on cubical meshes.

Recently, a methodology to extend the Crouzeix-Raviart element space on simplicial grids from $H\Lambda^0$ (namely $H^1$) to general $H\Lambda^k$ is presented \cite{Zhang.S2022padao}. The methodology is based on a reinterpretation of the Crouzeix-Raviart element space that, with $V^{\rm CR}_h$ and $\uV{}^{\rm RT}_h$ being respectively the lowest degree Crouziex-Raviart and Raviart-Thomas element spaces, on a simplicial triangulation $\mathcal{T}_h$,
\begin{equation}\label{eq:propcr}
V^{\rm CR}_h=\left\{v_h\in L^2(\Omega):v_h|_T\in P_1(T),\ \forall\,T\in\mathcal{T}_h,\ \mbox{and}\  \sum_{T\in\mathcal{T}_h}\int_T\nabla v_h\utau{}_h+\int_Tv_h\dv\utau{}_h=0\ \forall\,\utau{}_h\in \uV{}^{\rm RT}_{h0}\right\}.
\end{equation}
Namely, the far well known discrete Green's formula
\begin{equation}
\sum_{T\in\mathcal{T}_h}\int_T\nabla v_h\utau{}_h+\int_Tv_h\dv\utau{}_h=0,\ \ \mbox{for}\ v_h\in\ V^{\rm CR}_h\ \mbox{and}\ \utau{}_h\in \uV{}^{\rm RT}_{h0}
\end{equation}
is not only a property of $V^{\rm CR}_h$ and $\uv{}^{\rm RT}_{h0}$, but also it can work as the definition of $V^{\rm CR}_h$ given $\uV^{\rm RT}_{h0}$. Specifically, $V^{\rm CR}_h$ for $H^1(\Omega)$ can be viewed as a discrete adjoint space of $\uV^{\rm RT}_{h0}$ for $H_0(\dv,\Omega)$. This hints a general methodology to define nonconforming finite element spaces by Whitney forms. Let $\fW^*_{h0}\Lambda^k$ denote the standard conforming Whitney finite element space associated with $H^*_0\Lambda^k$. The nonconforming Whitney finite element space associated with $H\Lambda^k$ is defined in \cite{Zhang.S2022padao} as 
\begin{equation*}
\mathbf{W}^{\rm abc}_h\Lambda^k:=\left\{\fomega_h\in  \mathcal{P}^-_1\Lambda^k(\mathcal{T}_h): \sum_{T\in\mathcal{T}_h}\langle\fomega_h,\odelta_{k+1}\feta_h\rangle_{L^2\Lambda^k(T)}-\langle\od^k\fomega_h,\feta_h\rangle_{L^2\Lambda^{k+1}(T)}=0,\ \forall\,\feta_h\in\fW^*_{h0}\Lambda^{k+1}\right\},
\end{equation*}
with $\mathcal{P}^-_1\Lambda^k(\mathcal{T}_h)$ being the space of piecewise Whitney forms on $\mathcal{T}_h$. The uniform discrete Poincar\'e inequality, discrete Helmholtz decomposition, discrete Hodge decomposition, discrete Poincar\'e-Leftschitz duality, discrete de Rham complexes and commutative diagrams are established in \cite{Zhang.S2022padao} with locally defined discrete operators by inheriting adjoint relationships. This initializes a new way to impose global continuity condition onto local functions different from existing finite element methods and hints several other low(est)-degree finite element spaces\cite{Zhang.S2022pfemdde,Zhang.S2022pfemhl,Zhang.S2022ldHR2D}. All these spaces are not only theoretical constructions but also implementable for finite element methods. 
~\\

In this paper, applying the approach of \cite{Zhang.S2022padao}, we construct a unified family of finite element spaces by Whitney forms on cubical meshes for $H\Lambda^k$ in $\rn$. Specifically, given a cubical mesh $\mathcal{G}_h$, let, with precise meaning given later, $\mathcal{P}^-_1\Lambda^k(\mathcal{G}_h)$ be the space of piecewise Whitney forms on $\mathcal{G}_h$ and $V^{\q,*}_{h0}\Lambda^{k+1}$ be the standard conforming finite element space for $H^*_0\Lambda^{k+1}$ by $\star(\q^-_1\Lambda^{k+1})$. We define the Whitney form space for $H\Lambda^k$ by
\begin{equation*}
\fW^{\rm def}_h\Lambda^k:=\left\{\fomega_h\in\mathcal{P}^-_1\Lambda^k(\mathcal{G}_h):\sum_{K\in\mathcal{G}_h}\langle\od^k\fomega_h,\fmu_h\rangle_{L^2\Lambda^{k+1}(K)}-\langle\fomega_h,\odelta_{k+1}\fmu_h\rangle_{L^2\Lambda^k(K)}=0,\ \forall\,\fmu_h\in V^{\q,*}_{h0}\Lambda^{k+1}\right\}.
\end{equation*}
They are the extension of the Crouzeix-Raviart's element on cubical meshes, while Park-Sheen\cite{Park.C;Sheen.D2003}'s space is included in the family as one for $H\Lambda^0$; it is easy to see $V^{\rm PS}_h$ can be interpreted the same way as \eqref{eq:propcr}. Based on the spaces, discrete de Rham complexes and compatible finite element discretizations for $H\Lambda^k$ elliptic problems can be constructed. 

Again, the space $\fW^{\rm def}_h\Lambda^k$ does not coincide with a ``finite element" in Ciarlet's sense, and therefore new techniques are needed to construct the approximation. Technically, the main ingredient is to use the interpolation $\mathbb{I}^{\od^k}_K:H\Lambda^k(K)\to \mathcal{P}^-_1\Lambda^k(K)$ defined in \cite{Zhang.S2022padao} which possess the optimal approximation in piecewise $H\Lambda^k$ norm, and prove that the interpolator maps the functions of $V^{\q}_h\Lambda^k$ into $\fW^{\rm def}_h\Lambda^k$, and the approximation of $\fW^{\rm def}_h\Lambda^k$ can thus be proved. This routine is similar to Hu-Shi\cite{Hu.J;Shi.Z2005}'s argument, but carried out in a different way; actually, restricted to the $k=0$ case in $\mathbb{R}^2$ which was studied in Hu-Shi\cite{Hu.J;Shi.Z2005}, $\mathbb{I}^{\od^k}_K$ is different from the interpolator used there. It actually confirms that $\fW^{\rm def}_h\Lambda^k$ is implementable. The contents of this present paper are self-contained. 
~\\

Finally we remark that this paper is relevant to the question raised in Christiansen-Gillette\cite{Christiansen.S;Gillette.A2016} how to construct minimal spaces of differential forms equipped with commuting interpolators and contain prescribed functions. The minimality of several families of elements were proved in \cite{Christiansen.S;Gillette.A2016}. The major difference between the construction of this present paper and the elements therein lies in that, the condition of continuity of $\fW^{\rm def}_h\Lambda^k$ across the interfaces among cells are not necessarily local degrees of freedom of the Whitney forms $\mathcal{P}^-_1\Lambda^k$, and it is worthy of mentioning that the notion of nonconforming finite element space is necessary. We further refer to, e.g., \cite{Fortin.M;Soulie.M1983,Zhang.S2020IMA,Zhang.S2021SCM,Liu.W;Zhang.S2022jsc} for some existing nonconforming finite element spaces not of Ciarlet's type. 
~\\

The remaining of the paper is organized as follows. Section \ref{sec:pre} collects some useful preliminaries and notations. In Section \ref{sec:q-revisited}, we revisit the space $\q^-_1\Lambda^k$ and present some structural properties. In Section \ref{sec:whitrec}, we construct the finite element space by Whitney forms on cubical meshes, present the compatible discretization to $H\Lambda^k$-elliptic problems and the discrete de Rham complexes and commutative diagrams. Finally, some concluding remarks are given in Section \ref{sec:conc}.

\section{Preliminaries}
\label{sec:pre}

\subsection{$L^2$ theory of exterior calculus}
 
Following \cite{Arnold.D2018feec}, we denote by $\Lambda^k(\Xi)$ the space of differential $k$-forms on an $n$-dimensional domain $\Xi$, and $L^2\Lambda^k(\Xi)$ consists of differential $k$-forms with coefficients in $L^2(\Xi)$ component by component. $L^2\Lambda^k(\Xi)$ is a Hilbert space with inner product $\langle\cdot,\cdot\rangle_{L^2\Lambda^k(\Xi)}$. The exterior differential operator $\od^k:\Lambda^k(\Xi)\to \Lambda^{k+1}(\Xi)$ is an unbounded operator from $L^2\Lambda^k(\Xi)$ to $L^2\Lambda^{k+1}(\Xi)$. Denote, 
$$
H\Lambda^k(\Xi):=\left\{\fomega\in L^2\Lambda^k(\Xi):\od^k\fomega\in L^2\Lambda^{k+1}(\Xi)\right\},\ \ \ 0\leqslant k\leqslant n-1,
$$
and by $H_0\Lambda^k(\Xi)$ the closure of $\mathcal{C}_0^\infty\Lambda^k(\Xi)$ in $H\Lambda^k(\Xi)$.

The Hodge star operator $\star$ maps $L^2\Lambda^k(\Xi)$ isomorphically to $L^2\Lambda^{n-k}(\Xi)$ for each $0\leqslant k\leqslant n$. The \emph{codifferential operator} $\odelta_k\fmu=(-1)^{kn}\star\od^{n-k}\star\fmu$ is unbounded from $L^2\Lambda^k(\Xi)$ to $L^2\Lambda^{k-1}(\Xi)$.  Denote 
$$
H^*\Lambda^k(\Xi):=\left\{\fmu\in L^2\Lambda^k(\Xi):\odelta_k\fmu\in L^2\Lambda^{k-1}(\Xi)\right\},\ \ \ 1\leqslant k\leqslant n,
$$
and $H^*_0\Lambda^k(\Xi)$ the closure of $\mathcal{C}_0^\infty\Lambda^k(\Xi)$ in $H^*\Lambda^k(\Xi)$. Then $H^*\Lambda^k(\Xi)=\star H\Lambda^{n-k}(\Xi)$, and $H^*_0\Lambda^k(\Xi)=\star H_0\Lambda^{n-k}(\Xi)$. Further $\mathcal{N}(\od^k,H\Lambda^k)=\star \mathcal{N}(\odelta_{n-k},H^*\Lambda^{n-k})$, and $\mathcal{R}(\od^{k-1},H\Lambda^{k-1})=\star\mathcal{R}(\odelta_{n-k+1},H^*\Lambda^{n-k+1})$. Here and in the sequel of the paper, we use $\N$ and $\R$ to denote the null space and the range of certain operators. Namely, for example, $\N(\oT,\xD)=\left\{\xv\in\xD:\oT\xv=0\right\}$, and $\R(\oT,\xD)=\left\{\oT\xv:\xv\in\xD\right\}.$

\begin{lemma}\cite{Arnold.D2018feec}\label{lem:a-book-6.5}
The adjoint of 
\begin{enumerate}
\item $(\od^k,H\Lambda^k(\Xi))$ is $(\odelta_{k+1},H^*_0\Lambda^{k+1}(\Xi))$;
\item $(\od^k,H_0\Lambda^k(\Xi))$ is $(\odelta_{k+1},H^*\Lambda^{k+1}(\Xi))$;
\item $(\odelta_{k+1},H^*\Lambda^{k+1}(\Xi))$ is $(\od^k,H_0\Lambda^k(\Xi))$;
\item $(\odelta_{k+1},H^*_0\Lambda^{k+1}(\Xi))$ is $(\od^k,{H}\Lambda^k(\Xi))$.
\end{enumerate}
\end{lemma}
In the sequel, we will drop the domain $``\Xi"$ when there is no ambiguity brought in.

\subsection{Basics of Whitney $k$-forms}

For an integer $n$, denote the set of $k$-indices, $k\leqslant n$, as
$$
\mathbb{IX}_{k,n}:=\left\{\ixalpha=(\ixalpha_1,\dots,\ixalpha_k)\in\mathbb{N}^k:1\leqslant \ixalpha_1<\ixalpha_2<\dots<\ixalpha_k\leqslant n,\ \mathbb{N}\ \mbox{the\ set\ of\ integers}\right\}. 
$$
For $\ixalpha\in\mathbb{IX}_{k,n}$, we denote its cardinal number $|\ixalpha|=k$. We make the convention that, if $\ixalpha$ is empty, $|\ixalpha|=0$. The Koszul operator $\okappa$ is such that
$$
\okappa(\dx^{\ixalpha_1}\wedge\dots\wedge\dx^{\ixalpha_k})=\sum_{j=1}^k(-1)^{j+1}x_{\ixalpha_j}\dx^{\ixalpha_1}\wedge\dots\wedge\dx^{\ixalpha_{j-1}}\wedge\dx^{\ixalpha_{j+1}}\wedge\dots\wedge\dx^{\ixalpha_k}.
$$

Following \cite{Arnold.D;Falk.R;Winther.R2006acta}, denote by $\mathcal{P}_r\Lambda^k$ the space of differential $k$-forms with polynomial coefficients of degree of at most $r$. The space of Whitney forms, the lowest-degree trimmed polynomial $k$-forms, associated with the operator $\od^k$ can be denoted by (\cite{Arnold.D;Falk.R;Winther.R2006acta,Arnold.D;Falk.R;Winther.R2010bams,Arnold.D2018feec})
$$
\mathcal{P}^-_1\Lambda^k=\mathcal{P}_0\Lambda^k+\okappa(\mathcal{P}_0\Lambda^{k+1}).
$$
We denote the space of Whitney forms associated with the operator $\odelta_k$ by
\begin{equation}
\mathcal{P}^{*,-}_1\Lambda^k:=\star(\mathcal{P}^-_1\Lambda^{n-k})=\mathcal{P}_0\Lambda^k+\star\okappa\star(\mathcal{P}_0\Lambda^{k-1}).
\end{equation}
It is known that $\mathcal{P}^-_1\Lambda^0=\mathcal{P}_1\Lambda^0$ and $\mathcal{P}^-_1\Lambda^n=\mathcal{P}_0\Lambda^n$. Then $\mathcal{P}^{*,-}_1\Lambda^0=\mathcal{P}_0\Lambda^0$ and $\mathcal{P}^{*,-}_1\Lambda^n=\mathcal{P}_1\Lambda^n$.
The evident connection below is useful:
\begin{equation}\label{eq:localcouple}
\mathcal{R}(\od^k,\mathcal{P}^-_1\Lambda^k)=\mathcal{P}_0\Lambda^{k+1}=\N(\odelta_{k+1},\mathcal{P}^{*,-}_1\Lambda^{k+1}),\ \mbox{and}\ \ \mathcal{N}(\od^k,\mathcal{P}^-_1\Lambda^k)=\mathcal{P}_0\Lambda^k=\mathcal{R}(\odelta_{k+1},\mathcal{P}^{*,-}\Lambda^{k+1}).
\end{equation}

On a cubic $K$, following \cite{Zhang.S2022padao}, define the adjoint projection
\begin{equation}
\mathbb{I}^{\od^k}_K:H\Lambda^k(K)\to \mathcal{P}^-_1\Lambda^k(K)
\end{equation}
such that, for $\fmu\in \mathcal{P}^{*,-}_1\Lambda^k$,
\begin{equation}\label{eq:pwap}
\langle \od^k\mathbb{I}^{\od^k}_K\fomega, \fmu\rangle_{L^2\Lambda^{k+1}(K)}-\langle \mathbb{I}^{\od^k}_K\fomega,\odelta_{k+1} \fmu\rangle_{L^2\Lambda^k(K)}
=
\langle \od^k\fomega, \fmu\rangle_{L^2\Lambda^{k+1}(K)}-\langle \fomega,\odelta_{k+1} \fmu\rangle_{L^2\Lambda^k(K)}. 
\end{equation}
By \eqref{eq:localcouple}, $\mathbb{I}^{\od^k}_K$ is well-defined. 
Given a cubical mesh $\mathcal{G}_h$ on a domain $\Omega\subset \rn$, denote
\begin{equation}
\displaystyle\mathcal{P}^-_1\Lambda^k(\mathcal{G}_h):=\prod_{K\in\mathcal{G}_h}\mathcal{P}^-_1\Lambda^k(K), \ \mbox{and}\ \ \displaystyle\mathcal{P}^{*,-}_1\Lambda^k(\mathcal{G}_h):=\prod_{K\in\mathcal{G}_h}\mathcal{P}^{*,-}_1\Lambda^k(K).
\end{equation}

Define a global interpolator
$$
\mathbb{I}^{\od^k}_h:\prod_{T\in\mathcal{G}_h} H\Lambda^k(K)\to \mathcal{P}^-_1\Lambda^k(\mathcal{G}_h),\ \ \mbox{by}\ \ (\mathbb{I}^{\od^k}_h\fomega)|_K=\mathbb{I}^{\od^k}_K(\fomega|_K),\ \forall\,K\in\mathcal{G}_h.
$$
When $k=n$, $\mathbb{I}^{\od^n}_K$ is the $L^2$ projection to $\mathcal{P}_0\Lambda^n$ on $K$, and $\mathbb{I}^{\od^n}_h$ is the $L^2$ projection to $\mathcal{P}_0\Lambda^n(\mathcal{G}_h)$. 

Here and in the sequel, we use the subscript $``\cdot_h"$ to denote the dependence on the mesh. In particular, an operator with the subscript $``\cdot_h"$ indicates that the operation is performed cell by cell. For example, $(\od^k_h\fomega_h)|_K=\od^k(\fomega_h|_K)$ for $K\in\mathcal{G}_h$. We refer to \cite{Zhang.S2022padao} for the proof of the lemma below, which is quite straightforward by noting \eqref{eq:localcouple}. In \cite{Zhang.S2022padao}, the interpolator is originally defined on simplices; the virtue is the same on cubics.

\begin{lemma}\label{lem:globaloa}\cite{Zhang.S2022padao}
For $\displaystyle\fomega\in \prod_{K\in\mathcal{G}_h} H\Lambda^k(K)$, with $C_{k,n}$ depending on the shape regularity of $\mathcal{G}_h$, 
$$
\|\fomega-\mathbb{I}^{\rm \od^k}_h\fomega\|_{\od^k_h}\leqslant C_{k,n}\inf_{\feta_h\in \mathcal{P}^-_1\Lambda^k(\mathcal{G}_h)}\|\fomega-\feta_h\|_{\od^k_h}. 
$$ 
\end{lemma}
\noindent We use $\|\cdot\|_{\od^k_h}$ to denote the broken $H\Lambda^k$ norm on $\mathcal{G}_h$; specifically, $\displaystyle\|\cdot\|_{\od^k_h}=\sqrt{\sum_{K\in\mathcal{G}_h}\|\cdot\|_{H\Lambda^k(K)}^2}$.

\begin{lemma} [Commutative diagrams]\label{lem:localcd}\cite{Zhang.S2022padao}
For any $\fmu\in H\Lambda^k(K)$, $0\leqslant k\leqslant n-1$,
$\mathbb{I}_K^{\od^{k+1}}\od^k\fmu=\od^k\mathbb{I}_K^{\od^k}\fmu$.
\end{lemma}
It follows immediately $\mathbb{I}_h^{\od^{k+1}}\od^k\fmu=\od^k_h\mathbb{I}_h^{\od^k}\fmu$, for $\displaystyle \fmu\in\prod_{K\in\mathcal{G}_h}H\Lambda^k(K)$, 

\subsection{Tensor product space $\mathcal{Q}_1^-\Lambda^k$}
Following \cite{Arnold.D2018feec}, set $T=[-1,1]^n$, and the lowest-degree tensor product space is defined by
\begin{equation}
\q^-_1\Lambda^k(T):=\bigoplus _{\ixalpha\in \mathbb{IX}_{k,n}}\left[\bigotimes _{i=1}^n\mathcal{P}_{1-\delta_{i,\ixalpha}}([-1,1])\right]\dx^{\ixalpha_1}\wedge\dots\wedge\dx^{\ixalpha_k},
\end{equation}
where 
\begin{equation}
\delta_{i,\ixalpha}=\left\{\begin{array}{ll}1,&i\in\{\ixalpha_1,\dots,\ixalpha_k\}\\ 0,&\mbox{otherwise}
\end{array}\right.
\end{equation}
The definition is valid for general cubit the edges of which are parallel to axises. 

We denote by $V^{\q}_h\Lambda^k$ and $V^{\q}_{h0}\Lambda^k$ the conforming finite element spaces by $\q^-_1\Lambda^k$ associated with $H\Lambda^k$ and $H_0\Lambda^k$, respectively. It holds immediately 
\begin{equation}
\R(\od^k,V^{\q}_h\Lambda^k)\subset \N(\od^{k+1},V^{\q}_h\Lambda^{k+1}),\ \mbox{and}\ \R(\od^k,V^{\q}_{h0}\Lambda^k)\subset \N(\od^{k+1},V^{\q}_{h0}\Lambda^{k+1}).
\end{equation}

\subsection{Notations and symbols}
In the sequel, we use the following notations for short:
\begin{itemize}
\item for $\ixalpha\in \mathbb{IX}_{k,n}$, we use $i\in\ixalpha$ to denote that $\delta_{i,\ixalpha}=1$, 
\item for $\ixalpha\in \mathbb{IX}_{k,n}$ and $\ixbeta\in \mathbb{IX}_{k',n}$ we use $\ixbeta\subset \ixalpha$ if $k'\leqslant k$ and $\ixbeta_i\in\ixalpha$ for $1\leqslant i\leqslant k'$; particularly, $\ixbeta\subset\ixalpha$ if $|\ixbeta|=0$;
\item for $\ixalpha\in \mathbb{IX}_{k,n}$ and $\ixbeta\in \mathbb{IX}_{k,n}$, we denote for short
\begin{equation}
\dx^{\ixalpha}:=\dx^{\ixalpha_1}\wedge\dots\wedge \dx^{\ixalpha_k},\quad\mbox{and}\quad x_{\ixbeta}:=x_{\ixbeta_1}\dots x_{\ixbeta_k};
\end{equation}
particularly $x_{\ixbeta}=1$ if $|\ixbeta|=0$;
\item for $\ixalpha\in \mathbb{IX}_{k,n}$, we denote $\ixalpha^c\in \mathbb{IX}_{(n-k),n}$ such that $\delta_{i,\ixalpha}+\delta_{i,\ixalpha^c}=1,\ i=1,\dots,n.$
\end{itemize}


%
%
%
\section{Structure of $\q^-_1\Lambda^k$ revisited}
\label{sec:q-revisited}

We begin with an equivalent presentation of the space $\q^-_1\Lambda^k$. 
\begin{lemma}\label{lem:microq-}
$\q^-_1\Lambda^k={\rm span}\{x_{\ixtau}\dx^{\ixsigma}:|\ixsigma|=k,\ 0\leqslant|\ixtau|\leqslant n-k,\ \ixtau\subset\ixsigma^c\}$.
\end{lemma}
The proof is straightforward.

\begin{lemma}\label{lem:ecq-1}
$\N(\od^k,\q^-_1\Lambda^k)=\R(\od^{k-1},\q^-_1\Lambda^{k-1})$.
\end{lemma}
\begin{proof}
By Lemma \ref{lem:microq-}, we can rewrite  $\displaystyle\q^-_1\Lambda^k=\bigoplus_{j=0}^n\left[\mathcal{H}_j\Lambda^k\cap \q^-_1\Lambda^k\right]$, where $\mathcal{H}_r\Lambda^k$ is the space of the k-forms with homogeneous polynomial coefficients of degree $r$. Further, given $\fmu\in \q^-_1\Lambda^k$ and $\displaystyle\fmu=\sum_{j=0}^n\fmu_j$ with $\fmu_j\in \mathcal{H}_j\Lambda^k\cap \q^-_1\Lambda^k$, $\od^k\fmu=0$ if and only if $\od^k\fmu_j=0$ for $0\leqslant j\leqslant n$. It follows then
$$
\N(\od^k,\q^-_1\Lambda^k)=\bigoplus_{j=0}^n\N(\od^k,\q^-_1\Lambda^k\cap \mathcal{H}_j\Lambda^k).
$$

Now given $\fmu_j\in \q^-_1\Lambda^k\cap \mathcal{H}_j\Lambda^k$ such that $\od^k\fmu_j=0$, set $\fomega_j=\kappa\fmu_j$, then $\fomega_j\in \q^-_1\Lambda^{k-1}$ and $\od^{k-1}\fomega_j=(k+1)\fmu_j$. Namely, 
$$
\N(\od^k,\q^-_1\Lambda^k\cap \mathcal{H}_j\Lambda^k)\subset \R(\od^{k-1},\q^-_1\Lambda^{k-1})
$$
It follows further that
$$
\N(\od^k,\q^-_1\Lambda^k)=\bigoplus_{j=0}^n\N(\od^k,\q^-_1\Lambda^k\cap \mathcal{H}_j\Lambda^k)\subset \R(\od^{k-1},\q^-_1\Lambda^{k-1}).
$$
The proof is completed. 
\end{proof}

\begin{remark}
By homotopy formula, the space $\q^-_1\Lambda^k(K)$ can be written as
\begin{equation}\label{eq:decomq-1K}
\q^-_1\Lambda^k(K)=\R(\od^{k-1},\q^-_1\Lambda^{k-1})\oplus \okappa\R(\od^k,\q^-_1\Lambda^k).
\end{equation} 
\end{remark}

Denote 
\begin{equation}
\q^{*,-}_1\Lambda^k:=\star(\q^-_1\Lambda^{n-k}).
\end{equation}
It holds immediately that 
\begin{equation}\label{eq:microqstar-}
\q^{*,-}_1\Lambda^k={\rm span}\{x_{\ixtau'}\dx^{\ixsigma'}:|\ixsigma'|=k,\ 0\leqslant|\ixtau'|\leqslant k,\ \ixtau'\subset\ixsigma'\}
\end{equation}
and
\begin{equation}
\N(\od^k,\q^-_1\Lambda^k)=\star\N(\odelta_{n-k},\q^{*,-}_1\Lambda^{n-k}), \quad \R(\od^k,\q^-\Lambda^k)=\star\R(\odelta_{n-k},\q^{*,-}_1\Lambda^{n-k}).
\end{equation}
Further,
\begin{equation}\label{eq:esqstar-}
\N(\odelta_k,\q^{*,-}_1\Lambda^k)=\R(\odelta_{k+1},\q^{*,-}_1\Lambda^{k+1}).
\end{equation}

\begin{proposition}\label{prop:othorg}
Given $x_{\ixtau}\dx^{\ixsigma}\in\q^-_1\Lambda^k$ and $x_{\ixtau'}\dx^{\ixsigma'}\in\q^{*,-}_1\Lambda^k$, 
\begin{equation}\label{eq:ipnv}
\langle x_{\ixtau}\dx^{\ixsigma},x_{\ixtau'}\dx^{\ixsigma'}\rangle_{L^2\Lambda^k(T)}\neq 0\ \ \mbox{only\ if}\ \ |\ixtau|=|\ixtau'|=0,\ \mbox{and}\ \ixsigma=\ixsigma'.
\end{equation}
\end{proposition}
\begin{proof}
That $\langle x_{\ixtau}\dx^{\ixsigma},x_{\ixtau'}\dx^{\ixsigma'}\rangle_{L^2\Lambda^k(T)}\neq 0$ implies that $\ixsigma=\ixsigma'$, and $\int_Tx_{\ixtau}x_{\ixtau'}\neq0$; noting that $T=[-1,1]^n$ and $\int_Tx_i=0$, $1\leqslant i\leqslant n$, we obtain that $\ixtau=\ixtau'$. Since $\ixtau\subset\ixsigma^c$ and $\ixtau'\subset\ixsigma'$, it follows further that $|\ixtau|=|\ixtau'|=0$. The proof is completed. 
\end{proof}

Denote 
\begin{equation}
\okappa^\delta:=\star\circ\okappa\circ\star.
\end{equation}

\begin{remark}
It is interesting to note that, roughly speaking, given 
$$
\fomega\in{\rm span}\{x_{\ixtau}\dx^{\ixsigma}:|\ixtau|=m,|\ixsigma|=k\},
$$ 
\begin{equation}
\left\{
\begin{array}{ccl}
\od^k\fomega &\in&{\rm span}\{x_{\ixtau'}\dx^{\ixsigma'}:|\ixtau'|=m-1,|\ixsigma'|=k+1\},
\\
\okappa\fomega&\in&{\rm span}\{x_{\ixtau'}\dx^{\ixsigma'}:|\ixtau'|=m+1,|\ixsigma'|=k-1\},
\\
\odelta_k\fomega&\in&{\rm span}\{x_{\ixtau'}\dx^{\ixsigma'}:|\ixtau'|=m-1,|\ixsigma'|=k-1\},
\\
\okappa^\delta\fomega&\in&{\rm span}\{x_{\ixtau'}\dx^{\ixsigma'}:|\ixtau'|=m+1,|\ixsigma'|=k+1\}.
\end{array}
\right.
\end{equation}
\end{remark}

The lemma below reveals the connections among $\mathcal{P}^-_1\Lambda^k$, $\mathcal{P}^{*,-}_1\Lambda^k$, $\q^-_1\Lambda^k$ and $\q^{*,-}_1\Lambda^k$. 

\begin{lemma}\label{lem:ap=int}
It holds for $\fomega\in\q^-_1\Lambda^k$ that for $\fmu\in \q^{-,*}_1\Lambda^{k+1}$,
\begin{equation}\label{eq:ap=int}
\langle \od^k\mathbb{I}^{\od^k}_T\fomega, \fmu\rangle_{L^2\Lambda^{k+1}(T)}-\langle \mathbb{I}^{\od^k}_T\fomega,\odelta_{k+1} \fmu\rangle_{L^2\Lambda^k(T)}
=
\langle \od^k\fomega, \fmu\rangle_{L^2\Lambda^{k+1}(T)}-\langle \fomega,\odelta_{k+1} \fmu\rangle_{L^2\Lambda^k(T)}. 
\end{equation}
\end{lemma} 
\begin{proof}
It is easy to note that 
$$
\q^{-,*}\Lambda^{k+1}=\N(\odelta_{k+1},\q^{-,*}\Lambda^{k+1})\oplus \okappa^\delta(\R(\odelta_{k+1},\q^{-,*}\Lambda^{k+1})),
$$
$$
\N(\odelta_{k+1},\q^{-,*}\Lambda^{k+1})=\mathcal{P}_0\Lambda^{k+1}\oplus \left[\N(\odelta_{k+1},\q^{-,*}\Lambda^{k+1})\cap{\rm span}\left\{x_{\ixtau'}\dx^{\ixsigma'}:|\ixsigma'|={k+1},\ |\ixtau'|\geqslant 1,\ \ixtau'\subset\ixsigma'\right\}\right],
$$
and
\begin{multline*}
\okappa^\delta(\R(\odelta_{k+1},\q^{-,*}\Lambda^{k+1}))=\okappa^\delta\left(\mathcal{P}_0\Lambda^k\oplus\left[\R(\odelta_{k+1},\q^{-,*}\Lambda^{k+1})\cap {\rm span}\{x_{\ixtau'}\dx^{\ixsigma'}:|\ixsigma'|=k,\ |\ixtau'|\geqslant 1,\ \ixtau'\subset\ixsigma'\}\right]\right)
\\
=\okappa^\delta(\mathcal{P}_0\Lambda^k)\oplus \okappa^\delta \left(\R(\odelta_{k+1},\q^{-,*}\Lambda^{k+1})\cap {\rm span}\left\{x_{\ixtau'}\dx^{\ixsigma'}:|\ixsigma'|=k,\ |\ixtau'|\geqslant 1,\ \ixtau'\subset\ixsigma'\right\}\right)
\\
=\okappa^\delta(\mathcal{P}_0\Lambda^k) \oplus \left[\okappa^\delta \left(\R(\odelta_{k+1},\q^{-,*}\Lambda^{k+1})\right)\cap {\rm span}\left\{x_{\ixtau'}\dx^{\ixsigma'}:|\ixsigma'|={k+1},\ |\ixtau'|\geqslant 2,\ \ixtau'\subset\ixsigma'\right\}\right].
\end{multline*}
Therefore,
\begin{multline*}
\q^{-,*}\Lambda^{k+1}=\mathcal{P}_0\Lambda^{k+1}\oplus \okappa^\delta(\mathcal{P}_0\Lambda^k) 
\\
\oplus\left[\N(\odelta_{k+1},\q^{-,*}\Lambda^{k+1})\cap{\rm span}\{x_{\ixtau'}\dx^{\ixsigma'}:|\ixsigma'|={k+1},\ |\ixtau'|\geqslant 1,\ \ixtau'\subset\ixsigma'\}\right]
\\
\oplus  \left[\okappa^\delta (\R(\odelta_{k+1},\q^{-,*}\Lambda^{k+1}))\cap {\rm span}\left\{x_{\ixtau'}\dx^{\ixsigma'}:|\ixsigma'|={k+1},\ |\ixtau'|\geqslant 2,\ \ixtau'\subset\ixsigma'\right\}\right].
\end{multline*}

Now, given 
$$
\fmu\in \mathcal{P}_0\Lambda^{k+1}\oplus \okappa^\delta(\mathcal{P}_0\Lambda^k),
$$ 
\eqref{eq:ap=int} holds by definition. Then, as for $x_{\ixtau}\dx^{\ixsigma}\ \mbox{with}\ \ixsigma\in \mathbb{IX}_{k,n} \ \mbox{and}\ \ixtau\subset\ixsigma^c$ it holds by \eqref{eq:ipnv} that 
\begin{equation}\label{eq:yigevanish}
\langle \od^k (x_{\ixtau}\dx^{\ixsigma}),x_{\ixtau'}\dx^{\ixsigma'}\rangle_{L^2\Lambda^{k+1}(T)}=0,\ \mbox{for}\ x_{\ixtau'}\dx^{\ixsigma'}\ \mbox{with}\ \ixtau'\subset\ixsigma'\in \mathbb{IX}_{k+1,n}\ \mbox{and}\ |\ixtau'|\geqslant 1,
\end{equation}
it follows that, given 
$$
\fmu\in \left[\N(\odelta_{k+1},\q^{-,*}\Lambda^{k+1})\cap{\rm span}\{x_{\ixtau'}\dx^{\ixsigma'}:|\ixsigma'|={k+1},\ |\ixtau'|\geqslant 1,\ \ixtau'\subset\ixsigma'\}\right],
$$ 
both the left and right hands of \eqref{eq:ap=int} are zero, and \eqref{eq:ap=int} is verified. Finally, as for $x_{\ixtau}\dx^{\ixsigma}\ \mbox{with}\ \ixsigma\in \mathbb{IX}_{k,n} \ \mbox{and}\ \ixtau\subset\ixsigma^c$ it holds further, beside \eqref{eq:yigevanish}, that 
\begin{equation}
\langle \fomega,\odelta_{k+1}\fmu\rangle_{L^2\Lambda^k(T)}=0,\ \mbox{for}\ x_{\ixtau'}\dx^{\ixsigma'}\ \mbox{with}\ \ixtau'\subset\ixsigma'\in \mathbb{IX}_{k+1,n}\ \mbox{and}\ |\ixtau'|\geqslant 2,
\end{equation}
it follows that, again, given 
$$
\fmu\in \left[\okappa^\delta (\R(\odelta_{k+1},\q^{-,*}\Lambda^{k+1}))\cap {\rm span}\left\{x_{\ixtau'}\dx^{\ixsigma'}:|\ixsigma'|={k+1},\ |\ixtau'|\geqslant 2,\ \ixtau'\subset\ixsigma'\right\}\right],
$$ 
both the left and right hands of \eqref{eq:ap=int} are zero, and \eqref{eq:ap=int} is verified. 

The proof is completed. 
\end{proof}

\begin{remark}
For a general rectangle $K$ whose edges are parallel to the respective axises but which is not necessarily $[-1,1]^n$, we introduce 
$$
\tilde{x}_i:=x_i+b_i,\ \ \mbox{such\ that}\ \int_K\tilde{x}_i=0,
$$
and define a $K$-dependent Koszul operator $\okappa_K$ by 
$$
\okappa_K(\dx^{\ixalpha_1}\wedge\dots\wedge\dx^{\ixalpha_k})=\sum_{j=1}^k(-1)^{j+1}\tilde{x}_{\ixalpha_j}\dx^{\ixalpha_1}\wedge\dots\wedge\dx^{\ixalpha_{j-1}}\wedge\dx^{\ixalpha_{j+1}}\wedge\dots\wedge\dx^{\ixalpha_k}.
$$ 
Then with respect to a $K$-dependent presentation $\tilde{x}_{\ixtau}\dx^{\ixsigma}$, Lemmas \ref{lem:microq-} and \ref{lem:ecq-1}, \eqref{eq:microqstar-} and \eqref{eq:esqstar-}, and particularly Proposition \ref{prop:othorg} still hold, and Lemma \ref{lem:ap=int} can be proved the same way. 
\end{remark}


%
%
%
\section{Finite element spaces for $H\Lambda^k$ by Whitney forms}
\label{sec:whitrec}

On a cubical mesh $\mathcal{G}_h$, define finite element spaces for $H\Lambda^k$ by Whitney forms as
\begin{equation*}
\fW^{\rm def}_h\Lambda^k:=\left\{\fomega_h\in\mathcal{P}^-_1\Lambda^k(\mathcal{G}_h):\sum_{K\in\mathcal{G}_h}\langle\od^k\fomega_h,\fmu_h\rangle_{L^2\Lambda^{k+1}(K)}-\langle\fomega_h,\odelta_{k+1}\fmu_h\rangle_{L^2\Lambda^k(K)}=0,\ \forall\,\fmu_h\in V^{\q,*}_{h0}\Lambda^{k+1}\right\},
\end{equation*}
and 
\begin{equation*}
\fW^{\rm def}_{h0}\Lambda^k:=\left\{\fomega_h\in\mathcal{P}^-_1\Lambda^k(\mathcal{G}_h):\sum_{K\in\mathcal{G}_h}\langle\od^k\fomega_h,\fmu_h\rangle_{L^2\Lambda^{k+1}(K)}-\langle\fomega_h,\odelta_{k+1}\fmu_h\rangle_{L^2\Lambda^k(K)}=0,\ \forall\,\fmu_h\in V^{\q,*}_h\Lambda^{k+1}\right\}.
\end{equation*}

\subsection{Capacity of the finite element spaces}

The spaces $\fW^{\rm def}_h\Lambda^k$ and $\fW^{\rm def}_{h0}\Lambda^k$ are not defined as a finite element in Ciarlet's sense. Its capacity, though not self-evident, can be confirmed by the lemma below.

\begin{lemma}\label{lem:capaw}
$\fW^{\rm def}_h\Lambda^k\supset \R(\mathbb{I}^{\od^k}_h,V^{\q}_h\Lambda^k)$, and $\fW^{\rm def}_{h0}\Lambda^k\supset \R(\mathbb{I}^{\od^k}_h,V^{\q}_{h0}\Lambda^k)$. 
\end{lemma}
\begin{proof}
Given $\fomega_h\in V^{\q}_h\Lambda^k$, for any $\fmu_h\in V^{\q,*}_{h0}\Lambda^{k+1}$,
By Lemma \ref{lem:ap=int}, 
$$
\sum_{K\in\mathcal{G}_h}\langle \od^k\mathbb{I}^{\od^k}_h\fomega_h, \fmu_h\rangle_{L^2\Lambda^{k+1}(K)}-\langle \mathbb{I}^{\od^k}_h\fomega_h,\odelta_{k+1} \fmu_h\rangle_{L^2\Lambda^k(K)}
=
\sum_{K\in\mathcal{G}_h}\langle \od^k\fomega_h, \fmu_h\rangle_{L^2\Lambda^{k+1}(K)}-\langle \fomega_h,\odelta_{k+1} \fmu_h\rangle_{L^2\Lambda^k(K)}=0;
$$
namely, $\mathbb{I}^{\od^k}_h\fomega_h\in \fW^{\rm def}_h\Lambda^k.$ Similarly, 
$\R(\mathbb{I}^{\od^k}_h,V^{\q}_{h0}\Lambda^k)\subset \fW^{\rm def}_{h0}\Lambda^k$.
The proof is completed. 
\end{proof}

\begin{lemma}\label{lem:oawdef}
Given $\fomega\in H\Lambda^k$, 
\begin{equation}\label{eq:oahlambda}
\inf_{\fmu_h\in \fW^{\rm def}_h\Lambda^k}\|\fomega-\fmu_h\|_{\od^k_h}\leqslant C\left(\inf_{\feta_h\in V^{\q}_h\Lambda^k}\|\fomega-\feta_h\|_{\od^k}+\inf_{\fmu_h\in \mathcal{P}^-_1\Lambda^k(\mathcal{G}_h)}\|\fomega-\fmu_h\|_{\od^k_h}\right),
\end{equation}
and if $\fomega\in H_0\Lambda^k$, 
\begin{equation}\label{eq:oahlambda0}
\inf_{\fmu_h\in \fW^{\rm def}_{h0}\Lambda^k}\|\fomega-\fmu_h\|_{\od^k_h}\leqslant C\left(\inf_{\feta_h\in V^{\q}_{h0}\Lambda^k}\|\fomega-\feta_h\|_{\od^k}+\inf_{\fmu_h\in \mathcal{P}^-_1\Lambda^k(\mathcal{G}_h)}\|\fomega-\fmu_h\|_{\od^k_h}\right).
\end{equation}
\end{lemma}
\begin{proof}
Denote by $\fomega^{\q}$ the optimal approximation of $\fomega\in H\Lambda^k$ in $V^{\q}_h$. Then, by Lemma \ref{lem:capaw},
\begin{multline}
\inf_{\fmu_h\in \fW^{\rm def}_h\Lambda^k}\|\fomega-\fmu_h\|_{\od^k_h}
\leqslant \|\fomega-\mathbb{I}^{\od^k}_h\fomega^{\q}\|_{\od^k_h}\leqslant \|\fomega-\fomega^{\q}\|_{\od^k}+\|\fomega^{\q}-\mathbb{I}^{\od^k}_h\fomega^{\q}\|_{\od^k_h}
\\
\leqslant \|\fomega-\fomega^{\q}\|_{\od^k}+C\inf_{\fmu_h\in \mathcal{P}^-_1\Lambda^k(\mathcal{G}_h)}\|\fomega^{\q}-\fmu_h\|_{\od^k_h}\ \mbox{(by\ Lemma\ \ref{lem:globaloa})}
\\
\leqslant \|\fomega-\fomega^{\q}\|_{\od^k}+C\inf_{\fmu_h\in \mathcal{P}^-_1\Lambda^k(\mathcal{G}_h)}\left(\|\fomega^{\q}-\fomega\|_{\od^k}+\|\fomega-\fmu_h\|_{\od^k_h}\right)
\\
\leqslant C\left(\inf_{\feta_h\in V^{\q}_h\Lambda^k}\|\fomega-\feta_h\|_{\od^k}+\inf_{\fmu_h\in \mathcal{P}^-_1\Lambda^k(\mathcal{G}_h)}\|\fomega-\fmu_h\|_{\od^k_h}\right).
\end{multline}
This proves \eqref{eq:oahlambda}. Similarly can \eqref{eq:oahlambda0} be proved. The proof is completed. 
\end{proof}

\subsection{Compatible scheme to $H\Lambda^k$ problem}

Here we consider a typical elliptic variational problem: given $\ff\in L^2\Lambda^k$, find $\fomega\in H\Lambda^k$, such that 
\begin{equation}\label{eq:evpd}
\langle\od^k \fomega,\od^k \fmu\rangle_{L^2\Lambda^{k+1}}+\langle \fomega,\fmu\rangle_{L^2\Lambda^k}=\langle \ff,\fmu\rangle_{L^2\Lambda^k},\ \ \forall\,\fmu\in H\Lambda^k.
\end{equation}
It follows by Lemma \ref{lem:a-book-6.5} that $\od^k\fomega \in H^*_0\Lambda^{k+1}$, and $\odelta_{k+1}\od^k\fomega+\fomega=\ff$.

We consider the finite element discretization of \eqref{eq:evpd}: to find $\fomega_h\in\fW^{\rm def}_h\Lambda^k$, such that 
\begin{equation}\label{eq:dispro}
\langle\od^k_h \fomega_h,\od^k_h \fmu_h\rangle_{L^2\Lambda^{k+1}}+\langle \fomega_h,\fmu_h\rangle_{L^2\Lambda^k}=\langle \ff,\fmu_h\rangle_{L^2\Lambda^k},\ \ \forall\,\fmu_h\in \fW^{\rm def}_h\Lambda^k.
\end{equation}

Immediately \eqref{eq:evpd} and \eqref{eq:dispro} are well-posed. 

\begin{theorem}\label{thm:basicestrn}
Let $\fomega$ and $\fomega_h$ be the solutions of \eqref{eq:evpd} and \eqref{eq:dispro}, respectively. 
\begin{equation}\label{eq:errest}
\|\fomega-\fomega_h\|_{\od^k_h}\leqslant C\left(\inf_{\fmu_h\in\mathcal{P}^-_1(\mathcal{G}_h)}\|\fomega-\fmu_h\|_{\od^k_h}+\inf_{\feta_h\in V^{\q}_h\Lambda^k}\|\fomega-\feta_h\|_{\od^k}+\inf_{\fpsi_h\in V^{\q,*}_{h0}\Lambda^{k+1}}\|\od^k\fomega-\fpsi_h\|_{\odelta_{k+1}}\right).
\end{equation}
\end{theorem}
\begin{proof}
By Strang's lemma, 
\begin{equation}\label{eq:strang}
\|\fomega-\fomega_h\|_{\od^k_h}\leqslant 2\inf_{\fmu_h\in \fW^{\rm def}_h\Lambda^k}\|\fomega-\fmu_h\|_{\od^k_h}+\sup_{\fmu_h\in \fW^{\rm def}_h\Lambda^k}\frac{\langle\od^k\fomega,\od^k_h\fmu_h\rangle_{L^2\Lambda^{k+1}}-\langle\odelta_{k+1}\od^k\fomega,\fmu_h\rangle_{L^2\Lambda^k}}{\|\fmu_h\|_{\od^k_h}}.
\end{equation}
By the definition of $\fW^{\rm def}_h\Lambda^k$, for any $\fpsi_h\in V^{\q,*}_{h0}\Lambda^{k+1}$,
\begin{multline}\label{eq:constint}
\langle\od^k\fomega,\od^k_h\fmu_h\rangle_{L^2\Lambda^{k+1}}-\langle\odelta_{k+1}\od^k\fomega,\fmu_h\rangle_{L^2\Lambda^k}
\\
=\langle(\od^k\fomega-\fpsi_h),\od^k_h\fmu_h\rangle_{L^2\Lambda^{k+1}}-\langle\odelta_{k+1}(\od^k\fomega-\fpsi_h),\fmu_h\rangle_{L^2\Lambda^k}\leqslant \|\od^k\fomega-\fpsi_h\|_{\odelta_{k+1}}\|\fmu_h\|_{\od^k_h}.
\end{multline}
Then \eqref{eq:errest} follows by combining \eqref{eq:strang}, \eqref{eq:constint} and Lemma \ref{lem:oawdef}. The proof is completed. 
\end{proof}
Similar construction and result can be obtained for $H_0\Lambda^k$ spaces and problems. 

\begin{remark}
Concerning the practical implementation, we can use $\mathbb{I}^{\od^k}_h V^{\q}_h\Lambda^k$ at the place of $\fW^{\rm def}_h\Lambda^k$. Actually, according to the proof of Lemma \ref{lem:oawdef}, we can obtain 
\begin{equation}
\inf_{\fmu_h\in \R(\mathbb{I}^{\od^k}_h,V^{\q}_h\Lambda^k)}\|\fomega-\fmu_h\|_{\od^k_h}\leqslant C\left(\inf_{\feta_h\in V^{\q}_h\Lambda^k}\|\fomega-\feta_h\|_{\od^k}+\inf_{\fmu_h\in \mathcal{P}^-_1\Lambda^k(\mathcal{G}_h)}\|\fomega-\fmu_h\|_{\od^k_h}\right).
\end{equation}
The consistency error estimation holds simultaneously. In practice, given a set of basis functions of $V^{\q}_h\Lambda^k$, then its piecewise adjoint projection by $\mathbb{I}^{\od^k}_K$, cf. \eqref{eq:pwap}, even though not necessarily linearly independent, can work as a set of basis functions of $\mathbb{I}^{\od^k}_hV^{\q}_h\Lambda^k$.
\end{remark}

\begin{remark}
Similar construction and analysis can be given for $H_0\Lambda^k$ problems. 
\end{remark}

%
%
%
%
\subsection{Discrete complex and commutative diagram}

\begin{lemma}\label{lem:globales}
For $0\leqslant k\leqslant n-1$,
$\R(\od^k_h,\fW^{\rm def}_h\Lambda^k)\subset\N(\od^{k+1}_h,\fW^{\rm def}_h\Lambda^{k+1})$; $\R(\od^k_h,\fW^{\rm def}_{h0}\Lambda^k)\subset\N(\od^{k+1}_h,\fW^{\rm def}_{h0}\Lambda^{k+1})$.
\end{lemma}	
\begin{proof}
For $k=n-1$, the lemma holds as $\od^n_h\fomega=0$ with $\fomega\in\Lambda^n$. For $0\leqslant k\leqslant n-2$, given $\fmu_h\in \fW^{\rm def}_h\Lambda^k$, $\od^k_h\fmu_h\in\mathcal{P}_0\Lambda^{k+1}(\mathcal{G}_h)$. For any $\feta_h\in V^{\q,*}_{h0}\Lambda^{k+2}$,
\begin{multline*}
\underline{\langle\od^{k+1}_h(\od^k_h\fmu_h),\feta_h\rangle_{L^2\Lambda^{k+2}}}+\langle \od^k_h\fmu_h,\odelta_{k+2}\feta_h\rangle_{L^2\Lambda^{k+1}}
=\langle \od^k_h\fmu_h,\odelta_{k+2}\feta_h\rangle_{L^2\Lambda^{k+1}}
\\
= \langle \od^k_h\fmu_h,\odelta_{k+2}\feta_h\rangle_{L^2\Lambda^{k+1}}+\underline{\langle\fmu_h,\odelta_{k+1}(\odelta_{k+2}\feta_h)\rangle_{L^2\Lambda^k}}=0.
\end{multline*}
Here we use underline to label the vanishing terms, and have used the fact that $\R(\odelta_{k+2},V^{\q,*}_{h0}\Lambda^{k+2})\subset\N(\odelta_{k+1},V^{\q,*}_{h0}\Lambda^{k+1})$. It then follows that $\od^k_h\fmu_h\in \fW^{\rm def}_h\Lambda^{k+1}$, and $\od^k_h\fmu_h\in \N(\od^{k+1}_h,\fW^{\rm def}_h\Lambda^{k+1})$. 

Similarly, $\R(\od^k_h,\fW^{\rm def}_{h0}\Lambda^k)\subset\N(\od^{k+1}_h,\fW^{\rm def}_{h0}\Lambda^{k+1})$. The proof is completed. 
\end{proof}

We summarize Lemma \ref{lem:localcd}, Lemma \ref{lem:capaw} and Lemma \ref{lem:globales} into the theorem below. 
\begin{theorem} 
The discrete de Rham complexes and commutative diagrams below hold: 
\begin{equation}\label{eq:cdnb}
	\begin{array}{ccccccccccc}
	\mathbb{R} & ~~~\longrightarrow~~~ & V^{\q}_h\Lambda^0 & ~~~\xrightarrow{\od^0}~~~ & V^{\q}_h\Lambda^1 & ~~~\xrightarrow{\od^1}~~~ &  ... & ~~~\xrightarrow{\od^{n-1}}~~~ &  V^{\q}_h\Lambda^n & ~~~\xrightarrow{\od^n}~~~ &  0 \\
	& & \downarrow \mathbb{I}^{\od^0}_h & & \downarrow \mathbb{I}^{\od^1}_h & &&& \downarrow \mathbb{I}^{\od^n}_h  \\
	\mathbb{R} & \longrightarrow & \mathbf{W}^{\rm def}_h\Lambda^0 & \xrightarrow{\od^0_h} & \mathbf{W}^{\rm def}_h\Lambda^1 & \xrightarrow{\od^1_h} & ... & ~~~\xrightarrow{\od^{n-1}_h}~~~ & \mathbf{W}^{\rm def}_h\Lambda^n & ~~~\xrightarrow{\od^n_h}~~~ & 0
	\end{array};
\end{equation}

\begin{equation}\label{eq:cdwb}
	\begin{array}{ccccccccccccc}
 0 & ~~~\longrightarrow~~~ & V^{\q}_{h0}\Lambda^0 & ~~~\xrightarrow{\od^0}~~~ & V^{\q}_{h0}\Lambda^1 & ~~~\xrightarrow{\od^1}~~~ &  ... & ~~~\xrightarrow{\od^{n-1}}~~~ &  V^{\q}_{h0}\Lambda^n & ~~~\xrightarrow{\od^n}~~~ &  0 \\
	& & \downarrow \mathbb{I}^{\od^0}_h & & \downarrow \mathbb{I}^{\od^1}_h & &&& \downarrow \mathbb{I}^{\od^n}_h  \\
	0& \longrightarrow & \mathbf{W}^{\rm def}_{h0}\Lambda^0 & \xrightarrow{\od^0_h} & \mathbf{W}^{\rm def}_{h0}\Lambda^1 & \xrightarrow{\od^1_h} & ... & ~~~\xrightarrow{\od^{n-1}_h}~~~ & \mathbf{W}^{\rm def}_{h0}\Lambda^n & ~~~\xrightarrow{\od^n_h}~~~ & 0 
	\end{array}.
\end{equation}
\end{theorem}
Evidently, provided any commutative diagram associated with $\left\{V^{\q}_h\Lambda^k\right\}$, there can be a combined commutative diagram to $\left\{\fW^{\rm def}_h\Lambda^k\right\}$.

\section{Concluding remarks}
\label{sec:conc}

In this paper,  finite element spaces $\fW^{\rm def}_h\Lambda^k$ are constructed for $H\Lambda^k$ on cubical meshes in $\rn$ equipped with commuting interpolators. This presents a confirming answer to the question if a compatible finite element space can be constructed by theoretically minimal local shape functions. Further, an explicit set of locally supported basis functions are given, which indicates that the scheme can be implemented by a same routine as that for a standard finite element scheme. We remark here that, no extra structure, such as composite grid or macro element, is assumed. Similar method and analysis can be constructed for $H^*\Lambda^k$ problems. 

For the capacity of the space $\fW^{\rm def}_h\Lambda^k$, it is proved that $\mathbb{I}^{\od^k}_hV^{\q}_h\subset \fW^{\rm def}_h\Lambda^k$. As explained before, the technique used here is different from Hu-Shi\cite{Hu.J;Shi.Z2005}. Though, it is still interesting to extend their trick to $H\Lambda^k$ problems, namely, to construct firstly finite element spaces of ``rotated $Q_1$ element" type and associated interpolators, and to check the orthogonality-type condition \eqref{eq:ap=int}. This will be studied in future endeavor. Further, it will be interesting to verify if the two spaces $\mathbb{I}^{\od^k}_hV^{\q}_h$ and $\fW^{\rm def}_h\Lambda^k$ are identical and to verify if $V^{\q}_h\Lambda^k$ is isomorphic to $\fW^{\rm def}_h\Lambda^k$. Boundary conditions may make an effect. This may be discussed in future.

Accompanied to the newly defined space by Whitney forms, a parallel presentation holds that the adjoint projection of $V^{\q}_1\Lambda^k$ to $\mathcal{P}^-_1\Lambda^k(\mathcal{G}_h)$ leads to a space whose continuity is figured out by the adjoint relation with $V^{\q,*}_h\Lambda^{k+1}$. We may treat this as intrinsic relation between $\q^-_1\Lambda^k$ and its star space. This kind of property can be investigated further in future. 

Now, the Crouziex-Raviart element space has been generalized to $H\Lambda^k$ problems on both simplicial meshes and cubical meshes. It will be interesting how to construct a Crouzeix-Raviart type space on a mixed mesh by both simplices and cubics. This will be studied in future.

\end{document}